\documentclass[a4paper]{amsart}
\usepackage{graphicx}
\usepackage{amssymb}
\usepackage{amsmath}
\usepackage{amsthm}
\usepackage{amscd}
\usepackage[all,2cell]{xy}

\usepackage[pagebackref,colorlinks]{hyperref}

\UseAllTwocells \SilentMatrices
\newtheorem{thm}{Theorem}[section]

\newtheorem{lem}[thm]{Lemma}

\newtheorem{prop}[thm]{Proposition}
\theoremstyle{definition}

\theoremstyle{remark}
\newtheorem{rem}[thm]{\bf Remark}
\numberwithin{equation}{section}

\begin{document}

\title[An explicit projective bimodule resolution of a Leavitt path algebra]{An explicit projective bimodule resolution of a Leavitt path algebra}

\author{Xiao-Wu Chen, Huanhuan Li, Zhengfang Wang}

\subjclass[2010]{16E05, 13D03, 16G20, 16E60}
\thanks{}
\keywords{projective bimodule resolution, Leavitt path algebra, quasi-free algebra, derivation, Hochschild cohomology}

\date{\today}

\begin{abstract}
We construct an explicit projective bimodule resolution for the Leavitt path algebra of a row-finite quiver. We prove that the Leavitt path algebra of a row-countable quiver has Hochschild cohomolgical dimension at most one, that is, it is quasi-free in the sense of Cuntz-Quillen. The construction of the resolution relies on an explicit derivation of the Leavitt path algebra.
\end{abstract}

\maketitle

\section{Introduction}

Let $k$ be a field and $Q$ be an arbitrary quiver. The notion of a Leavitt path algebra $L_k(Q)$ of $Q$ with coefficients in $k$ is introduced in \cite{AGGP, AMP, AA, AA08}, which are algebraic analogues
of the Cuntz-Krieger $C^*$-algebras. Recently, Leavitt path algebras attract a lot of attention. The study of Leavitt path algebras is related to noncommutative algebraic geometry \cite{Smi}, singularity categories \cite{CY, Li}, Steinberg algebras \cite{CFST} and symbolic dynamic systems \cite{ALPS, Chen, Haz}.

This paper concerns the Hochschild cohomological properties of Leavitt path algebras. Recall that a quiver $Q$ is row-finite if it does not contain an infinite emitter.  We construct an explicit projective bimodule resolution for the Leavitt path algebra of a row-finite quiver; see Theorem \ref{thm:1}.

The constructed resolution and the desingularization in \cite{AA08} imply the following main consequence. Recall that a quiver $Q$ is row-countable \cite{AR} provided that there are at most countable arrows starting at any fixed vertex.

\vskip 5pt

\noindent {\bf Theorem.}   \emph{Let $k$ be a field and $Q$ be a row-countable quiver. Then the Leavitt path algebra $L_k(Q)$ has  Hochschild cohomological dimension at most one.}

\vskip 5pt

 Recall that an algebra with Hochschild cohomological dimension at most one is called quasi-free in \cite{CQ}; such algebras are viewed as analogues of smooth manifolds. For example, path algebras are well known to be quasi-free. The above theorem implies that Leavitt path algebras of row-countable quivers are quasi-free. In particular, we recover the fact that $L_k(Q)$ is hereditary; see \cite{AMP}.

The construction of the projective bimodule resolution boils down to an explicit construction of a derivation for a Leavitt path algebra. Using this derivation, we obtain a direct summand of the relative bar resolution,  which is the promised projective bimodule resolution of the Leavitt path algebra.

It is well known that Hochschild cohomological properties are related to the center and the derivations. Indeed, the center and the derivations of a Leavitt path algebra are studied in \cite{AC, CMMS} and \cite{Lop}. The resolution obtained here might shed new light on these studies.

The paper is structured as follows. In Section \ref{section2}, we recall basic facts on derivations. In Section \ref{section3}, we construct a derivation of the Leavitt path algebra. In the final section, we construct the promised projective bimodule resolution.

Throughout this paper, we fix a field $k$.  By a $k$-algebra, we mean an associative $k$-algebra, which possibly does not contain a unit. We require that modules are unital.

\section{Derivations and differential forms}
\label{section2}
In this section, we recall basic facts on derivations. The references for universal derivations are \cite[Section 2.2]{Cohn} and \cite{CQ}.

Let $A$ be a $k$-algebra with enough idempotents. In other words, the algebra $A$ contains a set $\{e_i\; |\; i\in I\}$ of \emph{local units}, where the elements $e_i$'s are pairwise orthogonal nonzero idempotents with the following property: for each element $a\in A$, there exists a finite subset $J\subset I$ satisfying $(\sum_{j\in J}e_j)a=a=a(\sum_{j\in J}e_j)$. We set $S=\bigoplus_{i\in I}ke_i$, which is a subalgebra of $A$.

Let $M$ be an $A$-bimodule. Recall that a $k$-derivation $d\colon A\rightarrow M$ is a $k$-linear map satisfying the Leibniz rule
$$d(ab)=d(a)b+ad(b)$$
for each $a, b \in A$. The $k$-derivation $d$ is called an \emph{$S$-derivation} if $d(s)=0$ for all $s\in S$, or equivalently, $d$ is a morphism of $S$-bimodules. We denote by ${\rm Der}_S(A, M)$ the space of $S$-derivations.

Consider the multiplication map $m\colon A\otimes_S A\rightarrow A$, whose kernel is denoted by $\Omega^1_S(A)$. The $A$-bimodule $\Omega^1_S(A)$ is called the \emph{bimodule of relative differential $1$-forms}. An $S$-derivation $d\colon A\rightarrow \Omega$ is \emph{universal} provided that for each $A$-bimodule $M$ the following map is an isomorphism
\begin{align}\label{equ:uni-der}
{\rm Hom}_{A\mbox{-}A}(\Omega, M)\longrightarrow  {\rm Der}_S(A, M), \quad f\mapsto f\circ d.
\end{align}
Here, by ${\rm Hom}_{A\mbox{-}A}$ we mean the Hom space between $A$-bimodules.

Let us explain the construction of universal $S$-derivations. For this, we recall the \emph{relative bar resolution} of $A$
\begin{align}\label{equ:bar}
\cdots \longrightarrow  A\otimes_S A\otimes_S A\otimes_S A\stackrel{d_2}\longrightarrow A\otimes_S A\otimes_S A \stackrel{d_1}\longrightarrow A\otimes_S A\stackrel{m}\longrightarrow A\longrightarrow 0,
\end{align}
which is a projective bimodule resolution of $A$. Here, the differentials $d_n\colon A^{\otimes_S (n+2)}\rightarrow A^{\otimes_S (n+1)}$ are given by
$$d_n(a_0\otimes a_1\otimes \cdots \otimes a_{n+1})=\sum_{i=0}^{n} (-1)^i a_0\otimes \cdots \otimes a_{i}a_{i+1}\otimes \cdots \otimes a_{n+1}.$$
 Then we have the following isomorphisms.
\begin{align}\label{equ:hom-der}
{\rm Der}_S(A, M) &\simeq \{\theta\in {\rm Hom}_{A\mbox{-}A}(A\otimes_S A\otimes_S A, M)\; |\; \theta\circ d_2=0\}\\
                     &\simeq {\rm Hom}_{A\mbox{-}A}(\Omega_S^1(A), M) \nonumber
\end{align}
Here, the first isomorphism sends $d$ to $\theta$ such that $\theta(a_0\otimes a_1\otimes a_2)=a_0d(a_1)a_2$; the second isomorphism uses the fact that the cokernel of $d_2$ is isomorphic to $\Omega_S^1(A)$. Consequently, the $S$-derivation $\Delta\colon A\rightarrow \Omega_S^1(A)$ such that
\begin{align}\label{equ:uni}
\Delta(a)=a\otimes e_i-e_j\otimes  a
\end{align}
 for each $a\in e_jAe_i$ and $i, j\in I$, is a universal $S$-derivation; compare \cite[Chapter 2, Proposition 2.5]{Cohn} and \cite[Propositions 2.4 and 2.5]{CQ}.

Indeed,  the above discussion yields the following observation.

\begin{lem}\label{lem:uni}
Let $d\colon A\rightarrow \Omega$ be an $S$-derivation. Then $d$ is universal if and only if there is an isomorphism $\theta\colon \Omega\rightarrow \Omega^1_S(A)$ of $A$-bimodules such that $(\theta \circ d)(a)=a\otimes e_i-e_j\otimes a$ for each $a\in e_jAe_i$ and $i, j\in I$. \hfill $\square$
\end{lem}

The following well-known results will be used later.

\begin{lem}\label{lem:der}
Let $A$ be a $k$-algebra generated by a subset $X\subseteq A$, and  $M$ be an $A$-bimodule. Assume that $ d', d''\colon A\rightarrow M$ are $k$-derivations such that $d'(x)=d''(x)$ for each $x\in X$. Then $d'=d''$.
 \end{lem}

\begin{proof}
The result follows directly from the Leibniz rule and induction.
\end{proof}

\begin{lem}\label{lem:der2}
Let $A$ be a $k$-algebra, and $M$ be an $A$-bimodule with a $k$-derivation $d\colon A\rightarrow M$. Assume that $I\subseteq A$ is a two-sided ideal generated by a set $T$ of elements satisfying $IM=0=MI$. Suppose that $d(t)=0$ for each $t\in T$. Then $d(I)=0$. Consequently, $d$ induces a derivation $A/I\rightarrow M$ of the quotient algebra $A/I$.
\end{lem}

\begin{proof}
We claim that $d(at)=0=d(ta)$ for any $t\in T$ and $a\in A$. By the Lebiniz rule, we have $d(at)=ad(t)+d(a)t$. By assumption, we have $d(t)=0$ and $d(a)t=0$. Then we are done with $d(at)=0$. Similarly, we have the other equality.
\end{proof}

\section{Derivations of path algebras and Leavitt path algebras}
\label{section3}

In this section, we construct a derivation of a Leavitt path algebra. For this, we recall a well-known result on the derivations of path algebras.

Recall that a \emph{quiver} $Q=(Q_{0}, Q_{1}; s, t)$ consists of a set $Q_{0}$ of vertices, a set $Q_{1}$ of arrows and two maps $s, t: Q_{1}\xrightarrow []{}Q_{0}$, which associate to each arrow $\alpha$ its starting vertex $s(\alpha)$ and its terminating vertex $t(\alpha)$, respectively. A vertex $i$ of $Q$ is \emph{regular} provided that the set $s^{-1}(i)$ is nonempty and finite, or equivalently, it is neither a sink nor an infinite emitter. We denote by $Q_0^{\rm re}$ the subset of $Q_0$ formed by regular vertices.

For a given quiver $Q$, a path of length $n$ is a sequence $p=\alpha_{n}\cdots\alpha_{2}\alpha_{1}$ of arrows with
$t(\alpha_{j})=s(\alpha_{j+1})$ for $1\leq j\leq n-1$. The starting vertex of $p$, denoted by $s(p)$, is $s(\alpha_{1})$. The terminating vertex of $p$, denoted by $t(p)$, is $t(\alpha_{n})$. We identify an
arrow with a path of length one. We associate to each vertex $i\in Q_{0}$ a trivial path $e_{i}$ of length zero. Set $s(e_i)=i=t(e_i)$. We denote by $Q_n$ the set of paths of length $n$.

Let $k$ be a field. The \emph{path algebra} $kQ=\bigoplus_{n\geq 0}kQ_{n}$ has a basis given by all paths in $Q$, whose multiplication is given as follows: for two paths $p$ and $q$ satisfying $s(p)=t(q)$, the product $pq$ is their concatenation;
otherwise, we set the product $pq$ to be zero. Here, we write the concatenation of paths from right to left. For example,  $e_{t(p)}p=p=pe_{s(p)}$ for each path $p$.

We observe that $\{e_i\; |\; i\in Q_0\}$ is a set of local units in $kQ$. Set $S=\bigoplus_{i\in Q_0}ke_i\subseteq kQ$.

The following result is well known. For more about $k$-derivations of $kQ$, we refer to \cite{GL}.

\begin{prop}\label{prop:path}
Let $Q$ be an arbitrary  quiver. Then there is a projective bimodule resolution of $kQ$
$$0\longrightarrow kQ\otimes_S kQ_1\otimes_S kQ\stackrel{\delta}\longrightarrow kQ\otimes_S kQ\stackrel{m}\longrightarrow kQ\longrightarrow 0,$$
where $m$ is the multiplication map and the bimodule homomorphism $\delta$ is uniquely determined by
$$\delta(e_{t(\alpha)}\otimes \alpha \otimes e_{s(\alpha)})=\alpha\otimes e_{s(\alpha)}-e_{t(\alpha)}\otimes \alpha, \quad \mbox{ for each }\alpha\in Q_1.$$

Consequently, for each $kQ$-bimodule $M$, there is an isomorphism
\begin{align}\label{equ:der-p}
{\rm Der}_S(kQ, M)\stackrel{\sim}\longrightarrow {\rm Hom}_{S\mbox{-}S}(kQ_1, M)
\end{align}
sending $d$ to $d|_{kQ_1}$,  its restriction to $kQ_1$.
\end{prop}

\begin{proof}
The projective bimodule resolution is well known; see \cite[Chapter 2, Proposition 2.6]{Cohn}. It follows that $\Omega^1_S(kQ)\simeq kQ\otimes_S kQ_1\otimes_S kQ$. Then the isomorphism (\ref{equ:der-p}) follows from (\ref{equ:uni-der}).
\end{proof}

We denote by $\overline{Q}$ the double quiver of $Q$, which is obtained by adding for each arrow $\alpha\in Q_1$ a new arrow $\alpha^*$ in the opposite direction, that is,  $s(\alpha^*)=t(\alpha)$ and $t(\alpha^*)=s(\alpha)$. The added arrows $\alpha^*$ are called \emph{ghost arrows}.

The \emph{Leavitt path algebra} $L_k(Q)$ is by definition the quotient algebra of $k\overline{Q}$ by the two-sided ideal generated by the following set $T$
\begin{align}\label{equ:t}
\big\{\alpha\beta^*-\delta_{\alpha, \beta}e_{t(\alpha)}\; |\; \alpha, \beta\in Q_1 \mbox{ with }s(\alpha)=s(\beta)\}\cup \{\sum_{\{\alpha\in Q_1\; |\; s(\alpha)=i\}}\alpha^*\alpha-e_i\; |\; i\in Q_0^{\rm re}\big\}.
\end{align}
The elements in $T$ are known as the \emph{Cuntz-Krieger relations}.

If $p=\alpha_{n}\cdots\alpha_2\alpha_{1}$ is a path in $Q$ of length $n\geq 1$, we define $p^{*}=\alpha_{1}^*\alpha_2^*\cdots\alpha_{n}^*$. We have $s(p^*)=t(p)$ and $t(p^*)=s(p)$. For convention, we set $e_{i}^*=e_{i}$ for $i\in Q_{0}$. We observe that for paths $p, q$ in $Q$ satisfying $t(p)\neq t(q)$, $p^*q=0$ in $L_k(Q)$. Recall that the Leavitt path algebra $L_k(Q)$ is spanned by the following set
$$\big\{e_i, p, p^*, \gamma^*\eta\; |\; i\in Q_0, p, \gamma, \text{~and~} \eta \text{~are ~nontrivial paths in ~} Q \text{~with~} t(\gamma)=t(\eta)\big\};$$
see \cite[Corollary 3.2]{Tom}. In general, this set is not $k$-linearly independent. For an explicit basis of $L_k(Q)$, we refer to \cite[Theorem 1]{AAJZ}.

We observe that  $L_k(Q)$ has a set  $\{e_i\; |\; i\in Q_0\}$ of local units. Set $S=\bigoplus_{i\in Q_0} ke_i\subseteq L_k(Q)$.

\begin{prop}\label{prop:der-L}
Let $Q$ be an arbitrary  quiver. Then there is a unique $S$-derivation $$D\colon L_k(Q)\longrightarrow L_k(Q)\otimes_S L_k(Q)$$ such that $D(\alpha)=\alpha\otimes e_{s(\alpha)}$ and $D(\alpha^*)=-e_{s(\alpha)}\otimes \alpha^*$ for each $\alpha\in Q_1$.
\end{prop}

\begin{proof}
Set $L=L_k(Q)$. We view $L\otimes_S L$ as a $k\bar{Q}$-bimodule. Applying (\ref{equ:der-p}) to the double quiver $\bar{Q}$, we infer that there is a unique $S$-derivation $D\colon k\bar{Q}\rightarrow L\otimes_S L$ satisfying $D(\alpha)=\alpha\otimes e_{s(\alpha)}$ and $D(\alpha^*)=-e_{s(\alpha)}\otimes \alpha^*$ for each $\alpha\in Q_1$.

Recall from (\ref{equ:t}) the generating set $T$ of the ideal. By Lemma \ref{lem:der2}, it suffices to claim that $D(t)=0$ for each $t\in T$. For $\alpha, \beta\in Q_1$ with $s(\alpha)=s(\beta)$, we have
\begin{align*}
D(\alpha \beta^*-\delta_{\alpha, \beta}e_{t(\alpha)})&= D(\alpha) \beta^*+\alpha D(\beta^*)\\
                                                     &=\alpha\otimes \beta^*-\alpha\otimes \beta^*=0.
\end{align*}
For each regular vertex $i$, we have
\begin{align*}
D(\sum_{\{\alpha\in Q_1\; |\; s(\alpha)=i\}}\alpha^*\alpha-e_i)&= \sum_{\{\alpha\in Q_1\; |\; s(\alpha)=i\}} (D(\alpha^*)\alpha+ \alpha^* D(\alpha))\\
                                                     &=\sum_{\{\alpha\in Q_1\; |\; s(\alpha)=i\}} (-e_i\otimes \alpha^*\alpha+\alpha^*\alpha\otimes e_i)\\
                                                     &=-e_i\otimes e_i+e_i\otimes e_i=0,
\end{align*}
where the third equality uses the second Cuntz-Krieger relations. This proves the claim. We are done.
\end{proof}

\begin{rem}
 Let $p=\alpha_n\cdots \alpha_2\alpha_1$ be a path of length $n$ in $Q$. Then we have  $D(p)=p\otimes e_{s(p)}+\sum_{j=1}^{n-1}\alpha_n\cdots \alpha_{j+1}\otimes \alpha_j\cdots \alpha_1$ and $D(p^*)=-e_{s(p)}\otimes p^* - \sum_{j=1}^{n-1}\alpha^*_1\cdots \alpha^*_{j}\otimes \alpha^*_{j+1}\cdots \alpha^*_n$.
\end{rem}

\section{A projective bimodule resolution}
\label{section4}
In this section, we describe an explicit projective bimodule resolution for the Leavitt path algebra of a row-finite quiver.

Let $Q$ be a \emph{row-finite} quiver, that is, $Q$ does not contain an infinite emitter. Recall that $Q_0^{\rm re}$ denotes the set of regular vertices, that is, the set of non-sinks for a row-finite quiver $Q$.  Let $L=L_k(Q)$ be the Leavitt path algebra and set $S=\bigoplus_{i\in Q_0}ke_i\subseteq L_k(Q)$.

We identify $L\otimes_S L$ with $\bigoplus_{i\in Q_0} Le_i\otimes_k e_iL$. In particular, $\bigoplus_{i\in Q^{\rm re}_0} Le_i\otimes_k e_iL$ is an $L$-subbimodule of $L\otimes_S L$. We observe that the image of the $S$-derivation $D\colon L\rightarrow L\otimes_S L$ in Proposition \ref{prop:der-L} is contained in $\bigoplus_{i\in Q^{\rm re}_0} Le_i\otimes_k e_iL$. Hence, by abuse of notation, we have the following $S$-derivation
$$D\colon L\longrightarrow \bigoplus_{i\in Q^{\rm re}_0} Le_i\otimes_k e_iL.$$

The following main result states a projective bimodule resolution of $L=L_k(Q)$. The resolution seems to be somewhat surprising, as it is quite different from the one in Proposition \ref{prop:path}.

\begin{thm}\label{thm:1}
Let $Q$ be a row-finite quiver. Then there is a projective bimodule resolution of $L$
$$0\longrightarrow \bigoplus_{i\in Q_0^{\rm re}} Le_i\otimes_k e_i L\stackrel{\partial}\longrightarrow L\otimes_S L \stackrel{m}\longrightarrow L\longrightarrow 0,$$
where $m$ is the multiplication map and the bimodule homomorphism $\partial$ is uniquely determined by $$\partial(e_i\otimes e_i)=e_i\otimes e_i-\sum_{\{\alpha\in Q_1\; |\; s(\alpha)=i \}} \alpha^*\otimes \alpha,\quad \mbox{for all }i\in Q_0^{\rm re}.$$

In particular, the $S$-derivation $D\colon L\rightarrow \bigoplus_{i\in Q_0^{\rm re}} Le_i\otimes_k e_i L$ is universal.
\end{thm}

\begin{proof}
In the following diagram,  the middle row is the relative bar resolution of $L$ given in (\ref{equ:bar}), and the bimodule homomorphisms $\iota$ and $\pi$ are given by $\iota(e_i\otimes e_i)=\sum_{{\{\alpha\in Q_1\; |\; s(\alpha)=i \}}} \alpha^*\otimes \alpha \otimes e_i$ and $\pi(a_0\otimes a_1\otimes a_2)=a_0D(a_1)a_2$, respectively.

\[\xymatrix{
& 0\ar[d]\ar[r] & \bigoplus_{i\in Q_0^{\rm re}} Le_i\otimes_k e_i L \ar[d]^-{\iota}\ar[r]^-{\partial} & L\otimes_S L \ar@{=}[d] \ar[r]^-{m} & L \ar@{=}[d]   \ar[r] & 0\\
\cdots \ar[r]& L\otimes_S L \otimes_S L \otimes_S L \ar[d] \ar[r]^-{d_2} & L\otimes_S L \otimes_S L \ar[d]^-{\pi}\ar[r]^-{d_1} & L\otimes_S L  \ar@{=}[d] \ar[r]^-{m} & L \ar@{=}[d]  \ar[r] & 0\\
& 0\ar[r] & \bigoplus_{i\in Q_0^{\rm re}} Le_i\otimes_k e_i L \ar[r]^-{\partial} & L\otimes_S L \ar[r]^-{m} & L \ar[r] & 0
}\]

We claim that the above diagram commutes. Indeed, we observe that $\partial=d_1\circ \iota$, since they take the same values on the generators $e_i\otimes e_i$. We have $\pi\circ d_2=0$, since $D$ is an $S$-derivation. To see that $\partial\circ \pi=d_1$, we use the isomorphism (\ref{equ:hom-der}). Hence, it suffices to show that the two derivations $\partial\circ D$ and $\Delta$ coincide, where $\Delta: L\xrightarrow{} L\otimes_S L$ is defined by (\ref{equ:uni}). We apply Lemma \ref{lem:der}. Then it suffices to verify that $\partial \circ D(x)=\Delta(x)$ for each $x\in \{e_i, \beta, \beta^*\; |\; i\in Q_0, \beta\in Q_1\}$. But this is trivial, using the definition of $\partial$ and the Cuntz-Krieger relations $\beta\alpha^*=\delta_{\alpha, \beta}e_{t(\alpha)}=\alpha\beta^*$. This proves the claim.

We observe that $\pi\circ \iota={\rm Id}$. It follows that as a complex,  the upper row in the above diagram is a direct summand of the middle row. Since the relative bar resolution is exact, we have the required projective bimodule resolution. By Lemma~\ref{lem:uni}, the last statement follows from $\partial\circ D=\Delta$.
\end{proof}

\begin{rem}
(1) Let $M$ be a left $L$-module. Then tensering over $L$ with the projective bimodule resolution in Theorem \ref{thm:1}, we obtain a projective resolution of $M$. In particular, the projective dimension of $M$ is at most one, that is, the Leavitt path algebra $L$ is hereditary; see also \cite[Theorem 3.5]{AMP}. The obtained projective resolution of $M$ is usually not minimal; compare \cite{AMT}.

(2) The above projective bimodule resolution will be useful to study the Hochschild cohomological groups ${\rm HH}^0(L)$ and ${\rm HH}^1(L)$ of $L$. Recall that ${\rm HH}^0(L)$ is isomorphic to the center of $L$, and that ${\rm HH}^1(L)$ is related to ${\rm Der}_S(L, L)$. Indeed, the universal $S$-derivation yields a neat description of  ${\rm Der}_S(L, M)$ for any $L$-bimodule $M$:
$${\rm Der}_S(L, M)\simeq {\rm Hom}_{L\mbox{-}L}(\bigoplus_{i\in Q_0^{\rm re}} Le_i\otimes_k e_i L, M)\simeq \bigoplus_{i\in Q_0^{\rm re}} e_iMe_i.$$
We mention the related results in \cite{AC, CMMS} and \cite{Lop}.  We observe that if $Q$ is a finite quiver, then the  Hochschild cohomological groups of $L$, viewed as a differential graded algebra with trivial differential,  is related to the singular Hochschild cohomology groups \cite{Wang} of the corresponding algebra of radical square zero. This will be treated elsewhere.
\end{rem}

 Recall that a quiver $Q$ is \emph{row-countable} \cite{AR} provided that there are at most countable arrows starting at any vertex. For example, any countable quiver is row-countable.  Applying the desingularization in \cite{AA08}, we infer the following result, as stated in Introduction.

\begin{thm}
Let $Q$ be a row-countable quiver. Then the Leavitt path algebra $L_k(Q)$ has Hochschild cohomological dimension at most one, that is, it is quasi-free in the sense of \cite{CQ}.
\end{thm}

\begin{proof}
By \cite[Theorem 5.6]{AA08}, there exists a row-finite quiver $Q'$ such that the two Leavitt path algebras $L_k(Q)$ and $L_k(Q')$ are Morita equivalent. It is well known that Morita equivalent algebras have the same Hochschild cohomological dimension; here, we observe that the argument in \cite[\S 1.5.6]{Lod} applies to algebras with local units. Then the result follows immediately from Theorem \ref{thm:1} for $Q'$.
\end{proof}

\begin{rem}
We expect that for an arbitrary quiver $Q$, the Leavitt path algebra $L_k(Q)$ is still quasi-free. But the argument in this paper will not work, since the desingularization only exists for row-countable quivers; see \cite{AR}. We observe that $L_k(Q)$ might be realized as the direct limit of subalgebras, which are the Leavitt path algebras of certain row-countable sub quivers of $Q$; compare \cite[Proposition 2.7]{Goo}. By \cite[Theorem 2.3]{Oso},  one obtains  an upper bound on the Hochschild cohomological dimension  of $L_k(Q)$, which depends on the cardinality ${\rm sup}_{i\in Q_0} |s^{-1}(i)|$.
\end{rem}

\vskip 10pt

\noindent {\bf Acknowledgements.}\quad Chen is supported by the National Natural Science Foundation of China (Nos. 11522113 and  11671245), and Li acknowledges the support of Australian Research Council grant DP160101481. Li thanks Xiao-Wu Chen for his hospitality when she visited the University of Science and Technology of China,and  thanks Roozbeh Hazrat for inspiring  discussions.  We thank Gene Abrams for the reference \cite{AR} and helpful comments.

 {\footnotesize \noindent Xiao-Wu Chen\\
 Key Laboratory of Wu Wen-Tsun Mathematics, Chinese Academy of Sciences,\\
 School of Mathematical Sciences, University of Science and Technology of China,\\
  Hefei 230026, Anhui, PR China \\
  xwchen$\symbol{64}$mail.ustc.edu.cn}

\vskip 5pt

{\footnotesize \noindent Huanhuan Li\\ Centre for Research in Mathematics, Western Sydney University, Australia\\
h.li$\symbol{64}$westernsydney.edu.au}

\vskip 5pt

{\footnotesize \noindent Zhengfang Wang\\ Beijing International Center for Mathematical Research, Peking University, \\ No. 5 Yiheyuan Road Haidian District, Beijing 100871, PR China \\
 wangzhengfang$\symbol{64}$bicmr.pku.edu.cn}


\begin{thebibliography}{9999}


\bibitem{AA} {\sc G. Abrams, and G. Aranda Pino}, {\em The Leavitt path algebra of a graph}, J. Algebra {\bf 293} (2) (2005), 319--334.


\bibitem{AA08} {\sc G.  Abrams, and G. Aranda Pino}, {\em  The Leavitt path algebras of arbitrary graphs},
 Houston J. Math. {\bf 34} (2) (2008), 423--442.


   \bibitem{ALPS}  {\sc G. Abrams, A. Louly, E. Pardo, and C. Smith,}  {\em Flow invariants in the classification
of Leavitt path algebras}, J. Algebra {\bf 333} (2011),  202--231.


\bibitem{AMT} {\sc G. Abrams, F. Mantese, and A. Tonolo}, {\em Extensions of simple modules over Leavitt path algebras}, J. Algebra {\bf 431} (2014), 78--106.


\bibitem{AR} {\sc G. Abrams, and  K.M. Rangaswamy}, {\em Row-finite equivalents exist only for row-countable graphs}, Contemp. Math. {\bf 562} (2012), 1--10.



\bibitem{AAJZ} {\sc A. Alahmadi, H. Alsulami, S.K. Jain,  and E. Zelmanov}, {\em Leavitt path algebras of finite Gelfand-Kirillov dimension}, J. Algebra Appl. {\bf 11} (6) (2012), 6pp.




    \bibitem{AGGP} {\sc P. Ara, M.A. Gonzalez-Barroso, K.R. Goodearl,  and E. Pardo}, {\em  Fractional skew monoid rings}, J. Algebra {\bf 278} (1) (2004), 104--126.


\bibitem{AMP} {\sc P. Ara, M.A. Moreno,  and E. Pardo}, {\em Nonstable $\mathbf{K}$-theory for graph algebras},
Algebr. Represent. Theor. {\bf 10} (2) (2007), 157--178.

\bibitem{AC} {\sc G. Aranda Pino, and K. Crow}, {\em The center of a Leavitt path algebra}, Rev. Mat. Iberoam.  {\bf 27} (2) (2011), 621--644.


\bibitem{Chen} {\sc X.W. Chen}, {\em Irreducible representations of Leavitt path algebras}, Forum Math. {\bf 27} (2015), 549--574.

\bibitem{CY} {\sc X.W. Chen, and D. Yang}, {\em Homotopy categories, Leavitt path algebras and Gorenstein projective modules}, Inter. Math. Res. Not. IMRN {\bf 10} (2015), 2597--2633.


\bibitem{CFST}{\sc L.O. Clark, C. Farthing, A. Sims, and M. Tomforde}, {\em A groupoid generalisation of Leavitt path algebras}, Semigroup Forum {\bf 89} (2014), 501--517.


\bibitem{CMMS} {\sc L.O. Clark, D. Martin Barguero, C. Martin Gonzales, and M. Siles Molina}, {\em  Using the Steinberg algebra model to determine the center of any Leavitt path algebra}, arXiv: 1604.01079v1, 2016.

\bibitem{Cohn} {\sc P.M. Cohn}, Algebra, the Second Edition, Vol. {\bf 3}, John Wiley Sons, Chichester, New York, Bisbane, Toroto, Singapore, 1991.

\bibitem{CQ} {\sc J. Cuntz, and D. Quillen}, {\em Algebra extensions and nonsingularity}, J. Amer. Math. Soc. {\bf 8}(2) (1995), 251--289.

\bibitem{Goo} {\sc K. Goodearl}, {\em Leavitt path algebras and direct limits}, in ``Rings, Modules and Representations", Contemporary Mathematics series (2009), 165--188.


\bibitem{GL} {\sc L. Guo, and F. Li}, {\em  Structure of Hochschild cohomology of path algebras and differential formulation of Euler's polyhedron formula,} Asian J. Math. {\bf 18} (2014),  545--572.


\bibitem{Haz}   {\sc R. Hazrat}, {\em The dynamics of Leavitt path algebras}, J. Algebra {\bf 384} (2013), 242--266.


\bibitem{Li} {\sc H.H. Li}, {\em The injective Leavitt complex}, Algebr. Represent. Theor. (2017), DOI: 10.1007/s10468-017-9741-9.

\bibitem{Lod} {\sc J.L. Loday}, Cyclic Homology, Grundl. Math. Wiss. {\bf 301},  Berlin, Heidelberg, New York,
Springer, 1992.

\bibitem{Lop} {\sc V. Lopatkin}, {\em  Derivations of Leavitt path algebra}, arXiv: 1509.05075v19, 2017.

\bibitem{Oso} {\sc B.L. Osofsky}, {\em Upper Bounds on Homological Dimensions}, Nagoya Math. J. {\bf 32} (1968), 315--323.

\bibitem{Smi} {\sc S.P. Smith}, {\em Category equivalences involving graded modules over path algebras of quivers,} Adv. Math. {\bf 230} (2012), 1780--1810.


\bibitem{Tom} {\sc M. Tomforde,} {\em Uniqueness theorems and ideal structure for Leavitt path algebras, } J. Algebra {\bf 318} (2007), 270--299.

    \bibitem{Wang} {\sc Z.F. Wang}, {\em Singular Hochschild cohomology of radical square zero algebras}, arXiv:1511.08348v1, 2015.


\end{thebibliography}
\end{document}